\RequirePackage{fix-cm}
\documentclass[smallextended
]{svjour3}       
\smartqed  
\usepackage{graphicx}
\usepackage{mathptmx}   
\usepackage{latexsym} 
\usepackage[cp1251]{inputenc}
\usepackage{amsmath,amsbsy}
\usepackage{amsfonts,amssymb,mathrsfs,amscd}
\usepackage{verbatim}
\usepackage{eucal}
\usepackage{graphicx}
\usepackage{enumerate}
\usepackage[pdftex,unicode,colorlinks,linkcolor=blue,
citecolor=red,bookmarksopen,pdfhighlight=/N]{hyperref}

\journalname{}
\numberwithin{equation}{section}
\newtheorem{maintheo}{\bf Main Theorem}
\newtheorem{uniqtheo}{Uniqueness Theorem}
\newtheorem{theoA}{\bf Theorem A}
\newtheorem{theoB}{\bf Theorem B}

\newcommand{\const}{{\rm const}}
\renewcommand{\Re}{\rm Re \,}
\newcommand{\dd}{\,{\rm d}}
\newcommand{\RR}{\mathbb R}
\newcommand{\CC}{\mathbb C}
\newcommand{\NN}{\mathbb N}
\newcommand{\DD}{\mathbb D}

\newcommand{\rad}{{\rm rad}}
\newcommand{\supp}{{\rm supp}}

\DeclareMathOperator{\Zero}{Zero}
\DeclareMathOperator{\sbh}{sbh}
\DeclareMathOperator{\Int}{int}
\DeclareMathOperator{\clos}{clos}
\DeclareMathOperator{\Har}{har}
\DeclareMathOperator{\Hol}{Hol}
\DeclareMathOperator{\dsbh}{\text{$\delta${\rm -sbh}}}
\DeclareMathOperator{\rtrc}{\text{$\rho${\rm -trc}}}
\DeclareMathOperator{\trc}{trc}
\DeclareMathOperator{\Meas}{Meas}

\begin{document}

\title{Zeros of holomorphic functions
in the unit disk\\ and 
$\rho$-trigonometrically convex functions \thanks{The work was supported by a grant of the Russian Science Foundation (project no. 18-11-00002, first author), by grants of the Russian Foundation of Basic Research (projects no.~16-01-00024, 18-51-06002, second author)
}
}
\author{Bulat~N.~Khabibullin
\and Farkhat~B.~Khabibullin
}


\institute{Bulat N. Khabibullin
\at
Department of Mathematics and IT, Bashkir State University, Ufa, Russia 
\\
\email{khabib-bulat@mail.ru}           
\\
Farkhat B. Khabibullin
\\
\email{khabibullinfb@mail.ru}
}

\date{Received: date / Accepted: date}

\maketitle

\begin{abstract}
Let $M$\/ be a subharmonic function with Riesz measure $\mu_M$ on 
the unit disk $\mathbb D$ in the complex plane $\mathbb C$. Let $f$ be a nonzero holomorphic function on $\mathbb D$ such that  $f$ vanishes on ${\sf Z}\subset \mathbb D$, 
 and satisfies $|f| \leq \exp M$ on $\mathbb D$.  Then restrictions on the growth of
$\mu_M$ near the boundary of $D$ imply certain restrictions on the distribution of $\sf Z$.
We give a quantitative study of this phenomenon in terms of special non-radial  test functions constructed using $\rho$-trigonometrically convex functions.
\end{abstract}

\section{Introduction}
\label{intro}
We use our results  \cite{Kha91}--\cite{KhT17} on zero subsets of holomorphic functions of one and several variables. The origins of our research including previous results of other authors are also described in sufficient detail in \cite{Kha91}--\cite{KhT17}. Earlier $\rho$-trigonometrically  convex and its multidimensional versions, $\rho$-subspherical functions,  were used to study zero sets of entire functions with constraints on their growth in {\it the complex plane\/} $\CC$ \cite[\S~4]{Kha91}, \cite[Theorem 3.3.5]{Kh12} and in $\CC^n$ \cite[4.2, Theorem 4.2.7]{Kh12}, $1<n\in \NN:=\{1,2,\dots\}$, of holomorphic functions on the punctured complex plane $\CC_*:=\CC\setminus \{0\}$ \cite[5.1.6]{KhAR18}, as well as to study the distributions of Riesz measures 
$\mu_u:=\frac{1}{2\pi}\Delta u$ for subharmonic functions $u$ with constraints on their growth in $\CC$, and in $\RR^m$ \cite[\S\S~2--3]{Kha91}, $2< m\in \NN$, where $\RR$ is {\it the real line\/} in $\CC$, and $\Delta$ is the {\it Laplace operator\/} acting in the sense of the theory of distributions. We use $\rho$-trigonometrically convex functions to  study of zero sequences of holomorphic functions with restrictions to their growth in {\it the unit disk}   
\begin{subequations}\label{DS}
\begin{align}
\DD&:=\{z=re^{i\theta}\colon 0\leq r<1,\; \theta \in \RR \}\subset \CC;
\tag{\ref{DS}d}\label{DSd}\\
\RR_{+\infty}&=\RR\cup \{+\infty \}, \quad \RR_{\pm \infty}:=
\RR_{-\infty}\cup \RR_{+\infty}, \quad
\RR_*:=\RR\setminus \{0\},
\tag{\ref{DS}r}\label{DSr}
\\
\RR^+&:=\{x\in \RR \colon x\geq 0\}\subset \CC, \quad  
\RR_*^+:=\RR^+\setminus \{0\}.
\tag{\ref{DS}+}\label{DS+}
\end{align}
\end{subequations} 
Let $S\subset \CC$. By $\Hol(S)$, $\Har(S)$, $\sbh (S)$, $\dsbh(S):=\sbh (S)-\sbh (S)$, $C^k(S)$ for $k\in \NN\cup \{\infty\}$, 
we denote resp. the classes of {\it holomorphic,\/ harmonic, subharmonic, $\delta$-subharmonic\/ {\rm \cite[3.1]{KhR18}},  
continuously  $k$ times differentiable\/} functions $u$ on an open set $\mathcal O_u\supset S$. But $C(S)$ is the class of all continuous functions on $S$. We denote the function identically equal to resp. 
$-\infty$ or $+\infty$ on $S$ by the same symbols $-\infty$ or $+\infty$, and
\begin{equation*}
\sbh_*(S):=\sbh(S)\setminus \{\boldsymbol{-\infty} \}, \; 
\dsbh_*(S):=\dsbh(S)\setminus \{\boldsymbol{\pm\infty} \},  
\; \Hol_*(S):=\Hol(S)\setminus \{0\},
\end{equation*}
 where the symbol $0$ is used to denote the number zero, the origin, zero vector, zero function, zero measure, etc.
The {\it positiveness\/} is everywhere understood as $\geq 0$ according to the context.

\begin{definition}[{\rm \cite[{\rm 8.1}]{Levin96_1}, \cite{GM}}]\label{trc}
Let $\rho \in \RR_*^+$. A $2\pi$-periodic  function 
$h\colon \RR \to \RR$ is called a {\it $\rho$-trigonometrically convex function\/}  if 
\begin{equation}\label{dftrc}
h(\theta)\leq \frac{\sin \rho (\theta_2-\theta)}{\sin \rho (\theta_2-\theta_1)}h(\theta_1)+\frac{\sin \rho (\theta-\theta_1)}{\sin \rho (\theta_2-\theta_1)}h(\theta_2)
\quad\text{\it for all $\theta \in (\theta_1,\theta_2)$}
\end{equation}
and {\it for all\/ $\theta_1,\theta_2\in \RR$ such that\/} $0<\theta_2-\theta_1<\pi/\rho$. A function $h\colon \RR \to \RR$ is a $0$-trigonometrically convex function if 
$h\equiv \const \in \RR$. Further, the class of all $2\pi$-periodic  $\rho$-trigonometrically convex function on $\RR$ is denoted as $\rtrc$, 
\begin{equation}\label{trc+}
\rtrc^+:=\{h\in \rtrc \colon h\geq 0 \;\text{on $\RR$} \}, \; \trc:=\bigcup_{\rho\in \RR^+}\rtrc, \; \trc^+:=\bigcup_{\rho\in \RR^+}\rtrc^+.
\end{equation}
\end{definition}
 We recall some properties of $2\pi$-periodic  $\rho$-trigonometrically convex functions 
that can be found in the works \cite{Levin56},  \cite{Levin96_1}, \cite{Leontev}, \cite{GM}, \cite{Maergoiz}, \cite{Djrbashian}. 
\begin{enumerate}[{\rm (i)}]
\item\label{tr:i} If  $h\in \trc$, then $h\in C(\RR)$. 
\item\label{tr:ii}  If $h\in \rtrc$, then  $h^+:=\max\{0,h\}\in \rtrc^+$. 
\item\label{tr:iii} Let  $h\in C^{2}(\RR)$ be a $2\pi$-periodic function. 
$h\in \rtrc$ if and only if \begin{equation}\label{C2tr}
h''(\theta)+\rho^2h(\theta)\geq 0 \quad\text{for all  $\theta\in \RR$.} 
\end{equation}
\item\label{tr:iv} A $2\pi$-periodic continuous function $h$ belongs to the class $\rtrc$ if and only if $h''+\rho^2h\geq 0$ in the sense of the distribution theory.
 \item\label{tr:v}    A $2\pi$-periodic continuous function $h$ belongs to the class $\rtrc$  if and only if the function $z=re^{i\theta}\mapsto h(\theta)r^{\rho }$ is subharmonic on $\CC$. 
\item\label{tr:vi}  If $h\in \rtrc^+$ and $\rho\leq \rho'\in \RR^+$, then $h\in \rho'\text{-}\trc^+$, i.,e., $\rtrc^+ \subset \rho'\text{-}\trc^+$.
\item\label{tr:vii} If a sequence of functions $h_n\in \rtrc^+$, $n\in \NN$ is decreasing, then the function $h:=\lim_{n\to \infty}h_n$ belongs to the same class  $\rtrc^+$.
\end{enumerate}

\begin{example}  Let $\rho\in \RR^+$. The $2\pi$-periodic continuation of the function 
\begin{equation*}
h(\theta):=\begin{cases}
\cos \rho \theta,\quad &\text{if $|\theta|<\frac{\pi}{2\rho}$},\\
0,\quad &\text{if $|\theta|\geq\frac{\pi}{2\rho}$},
\end{cases}
\qquad \theta\in (-\pi,\pi], 
\end{equation*}  
belong to the class $\rtrc^+$.
\end{example}
\begin{example}  Let $S\subset \CC$ be a bounded subset. Then the support  function 
$k_S(\theta):=\sup_{s\in S} {\Re} (se^{-i\theta}) $
of $S$ belongs to the class $1\text{-}\trc$. If $0\in S$, then $k_S\in 1\text{-}\trc^+$. For $\rho\in \RR^+$, the $\rho$-support function of $\rho$-convex domain $S\subset \CC$ belongs to the class $\rtrc$, and this 
$\rho$-support function belongs to the class $\rtrc^+$, when
$0\in S$  \cite[Ch.~VI, 2.2]{Djrbashian}, \cite[Ch.~II, \S~3]{Leontev}, \cite[\S~9]{Maergoiz}.  
\end{example}

\begin{example}  Let $u$ be a subharmonic function on $\CC$, and $\limsup_{z\to \infty}\frac{u(z)}{|z|^{\rho}}<+\infty$.  Then its $\rho$-indicator function 
\begin{equation*}
h(\theta):=\limsup_{r\to +\infty}\frac{u(re^{i\theta})}{r^{\rho}}\, , \quad \theta\in \RR,
\end{equation*}
belongs to the class $\rtrc$. See   \cite{Levin56},  \cite{Levin96_1}, \cite{Leontev}, \cite{GM}, \cite{Maergoiz} for $u:=\log |f|$ with an entire function $f$ on $\CC$.  
\end{example}

Let $\mathcal O\subset \CC$ be an open subset, and let 
\begin{equation}\label{dfZ}
{\sf Z}:=\{{\sf z}_k\}_{k=1,2,\dots}, \quad {\sf z}_k:=r_ke^{i\theta_{k}}\in \mathcal O, \quad r_k:=|{\sf z}_k|\in \RR^+, \; \theta_k\in \arg {\sf z}_k \subset \RR,
\end{equation} 
be {\it a sequence on\/ $\mathcal O$ without limit points in\/ $\mathcal O$.} Some points ${\sf z}_k$ can repeat. It is also
possible that ${\sf Z}=\varnothing$ is empty.  We associate with each sequence $\sf Z$ the integer-valued
positive {\it counting measure\/} $n_{\sf Z}$ on $\mathcal O$ by setting
\begin{equation}\label{nZ}
n_{\sf Z}(S) := \sum_{{\sf z}_k\in S} 1,\quad  S\subset \mathcal O;
\end{equation}
$n_{\sf Z}(S)$ is the number of points ${\sf z}_k$ lying in $S$. We denote by the same symbol as
the sequence $\sf Z$ the function ${\sf Z}\colon z\mapsto n_{\sf Z}\bigl(\{z\}\bigr)$, $z\in \mathcal O$, the {\it divisor\/} of the sequence $\sf Z$. In particular, we have $\supp\, {\sf Z}:=\supp\, n_{\sf Z}$ for the {\it support\/} $\supp$; ${\sf Z}\subset D$  means  that $\supp\, {\sf Z}\subset D$;
$z\in {\sf Z}$ (resp., $z\notin {\sf Z}$) means the same as $z\in \supp\,{\sf Z}$ (resp., as $z\notin \supp\,{\sf Z}$).

Departing from the usual treatment of a sequence as a function of an integer or a positive integer variable we say that a sequence $\sf Z$ coincides with a sequence $\sf Z'$ or that they are equal (we write ${\sf Z}={\sf Z'}$) if for the associated divisors
we have ${\sf Z} (z)\equiv {\sf Z'}(z)$ for all $z \in \mathcal O$. In other words, we regard a point sequence as
a representative of the equivalence class containing the sequences in 
$\mathcal O$ with equal divisors. An embedding ${\sf Z}\subset {\sf Z'}$ means that ${\sf Z} (z) \leq {\sf Z'} (z)$ for all $z\in \mathcal O$. See \cite{Kh12} in detail.

We denote by $\Zero_f$ the {\it zero sequence\/} of the
function $f \in \Hol_*(\mathcal O)$ in $\mathcal O$ numbered with multiplicities taken into account. Then \cite[Theorem 3.7.8]{Rans}
\begin{equation}\label{muf}
n_{\Zero_f}\overset{\eqref{nZ}}{=}\frac{1}{2\pi} \Delta \log |f| 
\end{equation} 
is the Riesz measure of function $\log |f| \in \sbh_*(\mathcal O)$.

A function $f \in \Hol_*(\mathcal O)$ {\it vanishes on\/} $\sf Z$ if ${\sf Z} \subset 
\Zero_f$ (we write $f({\sf Z}) = 0$). The function $0\in \Hol(\mathcal O)$
vanishes on any sequence ${\sf Z}\subset \mathcal O$.
  
For $r\in \RR^+$ and $z\in \CC$, we set $D(z,r) := \bigl\{
z'\in \CC\colon|z' - z| < r\bigr\}$
(i.e., $D(z,r)$ is an open disk of radius $r$ centered at $z$), 
$D(r) := D(0,r)$;  $\overline D(z,r) := \bigl\{
z'\in \CC\colon|z' - z| \leq r\bigr\}$
(i.e., $D(z,r)$ is a closed disk of radius $r$ centered at $z$), 
$\overline D(r) := \overline D(0,r)$, $\overline D(0):=\{0\}$. 

The class of all Borel real measures, i.e., {\it charges,\/} on a Borel subset  $S\subset \CC$ is denoted by $\Meas (S)$,  and $\Meas^+(S) \subset \Meas(S)$ is the subclass of all positive {\it measures.}    For a charge 
$\mu \in \Meas(S)$, we let $\mu^+$, $\mu^-$ and $|\mu| := \mu^+ + \mu^-$ resp.
 denote its {\it upper, lower,\/} and {\it total}\/ variations.

Let $h\colon \RR\to \RR$ be a bounded  $2\pi$-periodic Borel function on $\RR$;  $\mu\in \Meas (\DD)$.  We define the {\it radial counting function\/}  $\mu^{\rad}(\cdot ;h)$ of charge $\mu$ with weight $h$ on $[0,1)$ as  
\cite[(3.1)]{Kha91}, \cite[(0.2)]{Kha99}
\begin{equation}\label{ntk}
\mu^{\rad}(r ;h):=\int_{\overline D(r)} h(\arg z) \dd \mu(z), \quad r\in [0,1).
\end{equation}
In particular, the function $\mu^{\rad}(r):=\mu^{\rad}(r ;1)$ with weight $h\equiv 1$ is the classical radial counting function of  $\mu$. If  $\sf Z$ is a sequence in  $\mathcal O\overset{\eqref{dfZ}}{:=}\DD$ then
\cite{GS}, \cite[(0.4)]{Kha91},  \cite[(0.2)]{Kha99}
\begin{equation}\label{Ztf}
n_{\sf Z}^{\rad}(r ;h)\overset{\eqref{nZ}}{:=}\sum_{|{\sf z}_k |\leq r} h(\arg {\sf z}_k), \quad r\in [0,1).
\end{equation}
Here and below, a reference mark over a symbol of (in)equality, inclusion, or more general binary relation, etc. means that this relation is somehow related to this reference. For $-\infty\leq r<R\leq +\infty$  always 
\begin{equation}\label{irR}
\int_r^R\ldots :=\int_{(r,R)} \ldots \; .
\end{equation} 

 A particular result of our investigation is the following
\begin{uniqtheo}
Let $M\in \sbh_*(\DD)$ be a subharmonic function with Riesz measure   $\mu_M:=\frac{1}{2\pi} \Delta M\in \Meas^+(\DD)$, and let ${\sf Z}\overset{\eqref{dfZ}}{=}\{r_ke^{i\theta_k}\}_{k=1,2,\dots} \subset\DD\overset{\eqref{dfZ}}{=}:\mathcal O$
be a sequence , $h\in \trc^+$,  and $g\colon \RR^+\to \RR^+$ be a convex function with  $g(0)=0$. 
 If  a function $f\in \Hol(\DD)$ vanishes on ${\sf Z}$, 
satisfies the inequality 
$|f|\leq \exp M$ on $\DD$,  and 
\begin{subequations}\label{cu}
\begin{align}
\int_{1/2}^{1}  g\bigl(2(1-t)\bigr) \dd {\mu}_M^{\rad}(t;h)\overset{\eqref{ntk}}{<}+\infty, 
\tag{\ref{cu}M}\label{cuM}
\\ \intertext{but}
\sum_{1/2< r_k<1} g(1-r_k) h(\theta_k)\overset{\eqref{dfZ}}=+\infty,
\tag{\ref{cu}Z}\label{cuZ}
\end{align}
\end{subequations}
then $f$ is the zero function, i.\,e., $f\equiv 0$ on\/ $\DD$.
\end{uniqtheo}
In the case $M = 0$ with  $\mu_M=0$, $g(x)\equiv x$, $x\in \RR^+$, and $h\equiv 1\in 0\text{-}\trc^+$, 
the condition \eqref{cuZ} contradicts the classical Blaschke condition $\sum_{k} (1-r_k)<+\infty$. 
So, the Nevanlinna theorem on the distribution of zeros of bounded holomorphic functions  shows that 
our Uniqueness Theorem is accurate in this case.

By $\const_{a_1,a_2,\dots} \in \RR$ we denote constants that, in general, depend on $a_1,a_2,\dots$ and, unless
otherwise specified, only on them; $\const^+_{\dots} \geq 0$. 

\begin{maintheo} Let $M\in \dsbh_*(\DD)$ and $u\in \sbh_*(\DD)$  are functions, resp., with Riesz charge $\mu_M:=\frac{1}{2\pi}\Delta M\in \Meas (\DD)$ and with Riesz measure $\mu_u:=\frac{1}{2\pi}\Delta u\in \Meas^+(\DD)$. Let  $\rho\in \RR^+$. If $u\leq M$ on $\DD$, then there exists a constant $C:=\const^+_{\rho, M, u}\geq 0$ such that the inequality 
\begin{equation}\label{uM}
\int_{1/2}^{1} g\Bigl(\frac{1-t}{t}\Bigr) \dd \mu^{\rad}_u(t;h)\overset{\eqref{ntk}}{\leq} 
\int_{1/2}^{1} g\Bigl(\frac{1-t}{t}\Bigr) \dd \mu^{\rad}_M(t;h)+C
\end{equation}
 holds for any
\begin{enumerate}
\item[{\rm [g]}]  convex function $g\colon \RR^+\to \RR^+$ with $g(0)=0$ and $g(1)\leq 1$,
\item[{\rm [h]}] $2\pi$-periodic $\rho$-trigonometrically  convex  function $h\colon 
\RR\to [0,1]$.
\end{enumerate}
In particular, if\/ ${\sf Z}$ is a sequence   from\/ \eqref{dfZ} with $\mathcal O:=\DD$, and there exists a function 
$f\in \Hol_*(\DD)$, $f({\sf Z})=0$, satisfying the inequality\/ $|f|\leq \exp M$ on $\DD$,  then there is a constant $C:=\const_{\rho,M,{\sf Z}}$ such that  
\begin{equation}\label{uMf}
\sum_{1/2< r_k<1} g\Bigl(\frac{1-r_k}{r_k}\Bigr) h(\theta_k) \leq 
\int_{1/2}^{1} g\Bigl(\frac{1-t}{t}\Bigr) \dd \mu^{\rad}_M(t;h)+C\quad
\text{for any\/ {\rm [g]}--{\rm [h]}.}
\end{equation}
\end{maintheo}
The cases $u=M$ and $M=\log |f|$, ${\sf Z}=\Zero_f$, show that the inequalities  \eqref{uM} and \eqref{uMf} uniform with respect to [h]--[g]  are optimal up to an additive constant $C$.

\section{Subharmonic test functions and their role}\label{ssec1_2} 

By $\CC_{\infty}:=\CC\cup  {\infty}$ we denote the one-point Alexandroff compactification of $\CC$.  
 For a subset $S \subset \CC_{\infty}$,  $\clos S$, $\Int S$, and $\partial S$ are 
{\it the closure,the  interior, and the boundary\/} of $S$ in $\CC_{\infty}$. 
A (sub)domain in $\CC_{\infty}$ is an open connected
subset in $\CC_{\infty}$. Let $S_0\subset S\subset \mathbb C_{\infty}$. If  the  closure $\clos\, S_0$ is a compact subset of $S$ in the topology induced on $S$  from $\mathbb \CC_{\infty}$, then the set $S_0$ is the {\it relatively compact subset} of $S$, and we write $S_0\Subset S$.  Let 
\begin{equation}\label{SD0D}
\varnothing \neq S\Subset D\subset \CC_{\infty},\quad\text{where $D\neq \CC_{\infty}$ is\/ {\it domain}.} 
\end{equation} 
For a function $v\colon D\setminus S\to \RR$ we write 
\begin{equation}\label{dD}
\lim_{\partial D}v=0,\quad\text{\it if\/ 
$\lim_{D\ni z'\to z}v(z')=0$ for all\/ $z\in \partial D$.}  
\end{equation}
By definition, put
\begin{subequations}\label{sbh}
\begin{align}
\sbh_0(D\setminus S) &:=\Bigl\{v\in \sbh(D\setminus S)\colon 
 \lim_{\partial D}v\overset{\eqref{dD}}{:=}0\Bigr\},
\tag{\ref{sbh}o}\label{sbh0}\\
\sbh_0^+(D\setminus S)&\overset{\eqref{sbh0}}{:=}\bigl\{v\in \sbh_0(D\setminus S)\colon v\geq 0\text{ on }D \bigr\}. 
\tag{\ref{sbh}+}\label{sbh0+}
\end{align}
\end{subequations}

\begin{definition}[{\rm \cite[Definition 1]{KhR18}}]\label{def:tesf} 
We say that a function $v\overset{\eqref{sbh0+}}{\in} \sbh^+_0 (D\setminus S)$ is a subharmonic {\it test function\/} on $D$ outward $S$ if the function $v$ is bounded in $D\setminus S$. The class of such functions $w$ will be denoted by $\sbh_0^{+}(D\setminus S; <+\infty)$. For $b\in \RR^+$, put
\begin{equation}\label{<b}
\sbh_0^{+}(D\setminus S; \leq b)\overset{\eqref{sbh0+}}{:=}\Bigl\{v\in \sbh_0^{+}(D\setminus S;<+\infty)\colon \sup_{D\setminus S} v\leq b\Bigr\}.
\end{equation}
Thus, \begin{equation*}
\sbh_0^{+}(D\setminus S; <+\infty)=\bigcup_{b\in \RR^+}\sbh_0^{+}(D\setminus S; \leq b).
\end{equation*}
\end{definition}
The main role will be played by the following 
\begin{theoA}[{\rm \cite[Main Theorem]{KhR18} 
for $\CC$, see also \cite[Main Theorem]{KhT15}, \cite{KhT16}--\cite{KhT17}}]  
Let  $ M\in \dsbh_* (D)$ be a $\delta$-subharmonic function with Riesz charge $\mu_M=\frac{1}{2\pi}\Delta M$, and  
\begin{equation}\label{SD0D+}
\varnothing \neq \Int S\subset S=\clos S \overset{\eqref{SD0D}}{\Subset} D\subset \CC_{\infty}\neq  D.
\end{equation}
Then for any point $z_0\in \Int S$ with $M(z_0)\in  \RR$, any number  $b\overset{\eqref{DS+}}{\in} \RR_*^+$,  any regular for the Dirichlet Problem\/ {\rm \cite[4]{Rans}}  domain $\widetilde{D}\subset \CC_{\infty}$ with the Green function $g_{\widetilde{D}}(\cdot , z_0)$ with a pole at $z_0$ which satisfies the conditions  $S\Subset \widetilde{D}\subset D$  and $\CC_{\infty}\setminus \clos \widetilde{D}\neq \varnothing$, any subharmonic function $u\in \sbh_* (D)$ satisfying the inequality $u\leq M$ on  $D$, and any subharmonic test function  $v\overset{\eqref{<b}}{\in}  \sbh_0^+(D\setminus S;\leq b) $ the following inequality holds:
\begin{equation}\label{mest+}
\widetilde{C} u(z_0) 	+\int_{D\setminus S}  v \,d {\mu}_u 		\leq	\int_{D\setminus S}  v \,d {\mu}_M	+\int_{\widetilde{D}\setminus S} v \,d {\mu}_M^-   +\widetilde{C}\, \overline{C}_M,
\end{equation}
where $\mu_u:=\frac{1}{2\pi}\Delta u$ is the Riesz measure of the function $u$, 
\begin{equation}\label{cz0C}
\widetilde{C}:=\const_{z_0,S,\widetilde{D},b}^+:= \frac{b} {\inf\limits_{z\in \partial S}  g_{\widetilde{D}}(z, z_0)}>0,
\end{equation}
and the value $+\infty$ is possible for the constant 
\begin{equation}\label{CM}
\overline{C}_M:=	\int_{\widetilde{D}\setminus \{z_0\}} g_{\widetilde{D}}(\cdot, z_0)  \dd {\mu}_M  +\int_{\widetilde{D}\setminus S} g_{\widetilde{D}}(\cdot, z_0)  \dd {\mu}_M^-  +M^+(z_0), 
\end{equation}
but for  $\widetilde{D}\Subset D$ this is a certain constant $\overline{C}_M\overset{\eqref{CM}}{=}\const_{z_0,S, \widetilde{D},M,D}^+<+\infty$.
\end{theoA}
We use the following  simplified version of Theorem A.
\begin{theoB} Under the agreements \eqref{SD0D}, and \eqref{SD0D+},
let $M{\in} \dsbh_*(D)$  be a function with Riesz charge $\mu_M\in \Meas(D)$.
Then, for any function $u\in \sbh_*(D)$ with Riesz measure $\mu_u$  satisfying the inequality 
$u \leq M$ on $D$,  we have the inequality 
\begin{equation}\label{mest}
\int_{D\setminus S}  v \dd {\mu}_u 		\leq	\int_{D\setminus S}  v \dd {\mu}_M	+C\quad
\text{for all $v\overset{\eqref{<b}}{\in}  \sbh_0^+(D\setminus S;\leq 1) $}, 
\end{equation}
where   a constant  $C:=\const_{D, S,u,M}^+\in \RR^+$ is independent of\/ $v\overset{\eqref{<b}}{\in}  \sbh_0^+(D\setminus S;\leq 1) $.
\end{theoB}
\begin{proof}  There exists always a point $z_0\in \Int S$ and $r_0\in \RR^+_*$ such that \cite[3.1]{KhR18} 
\begin{equation}\label{z0r0}
\begin{split}
D(z_0,r_0)\Subset \Int S, \quad u(z_0)&\neq -\infty, \quad M(z_0)\neq \pm\infty , 
\\ 
\Bigl|\int_{D(z_0,r_0)}\log &|z-z_0| \dd \mu_M\Bigr|<+\infty.
\end{split}
\end{equation}
There is always a regular for the Dirichlet Problem\/  domain $\widetilde{D}$ such that 
$\Int S\Subset \widetilde{D}\Subset D$ {\rm \cite[4]{Rans}}. The choice of such point $z_0$ 
and such domain $\widetilde{D} $ is predetermined solely by sets $S,D$. We choice $b:=1$.
Thus, $\widetilde{C}\overset{\eqref{cz0C}}{=}\const_{D,S}^+\in \RR^+$ is a constant  
depending only on $S$ and $D$. In view of \eqref{z0r0}, by the definition \eqref{CM}, the constant 
$\overline{C}_M\overset{\eqref{CM}, \eqref{z0r0}}{=}\const_{D,S,M}^+\in \RR^+$ is depending only 
on $D,S,M$. Hence the constant 
\begin{equation*}
C\overset{\eqref{mest+}}{:=}|\widetilde{C} u(z_0)|	+ 
|\mu_M|(\widetilde{D}\setminus S)+\widetilde{C}\, \overline{C}_M\\
\geq - \widetilde{C} u(z_0) + 
\int_{\widetilde{D}\setminus S} v \,d {\mu}_M^-   +\widetilde{C}\, \overline{C}_M,
\end{equation*}
depends only on $D,S, u,M$, i.\,e., $C=\const_{D,S,u,M}^+\in \RR^+$. So, \eqref{mest} follows from \eqref{mest+}.
\end{proof}
A method of constructing subharmonic test functions on $\DD$ outward $D(r)$ by means of $\rho$-trigonometrically  convex positive functions is given by the following 
\begin{proposition}\label{trc-stf} Let $h\overset{\eqref{trc+}}{\in} \rtrc^+$
be a $2\pi$-periodic $\rho$-trigonometrically convex positive function, and let $g\colon \RR^+\to \RR^+$ be a convex function with $g(0)=0$. We set 
\begin{equation}\label{rrho}
  \frac{1}{2}\leq r_{\rho}:=\max\Bigl\{\frac{1}{2}, 1-\frac{1}{\rho^2}\Bigr\}<1.
\end{equation} 
Then the function
\begin{equation}\label{gh}
z:=re^{i\theta}\mapsto g\Bigl(\frac{1-r}{r}\Bigr) h(\theta), \quad r\in (0, 1), \; \theta\in \RR,\; z\in \DD\setminus \{0 \},
\end{equation}     
belongs to the class\/  {\rm (see \eqref{<b})}
\begin{equation}\label{hsbh}
\sbh_0^+\bigl(\DD \setminus \overline D(r_{\rho}); \leq b_{\rho}\bigr), \text{ where }b_{\rho}:=g\Bigl(\frac{1-r_{\rho}}{r_\rho}\Bigr)\max_{\theta} h(\theta).
\end{equation}
\end{proposition}
\begin{proof} We use the properties \eqref{tr:i}--\eqref{tr:vii} of $\rho$-trigonometrically convex functions.

There is a decreasing sequence of convex positive functions $g_n \underset{n\to\infty}{\searrow} g$ on $\RR$ such that $g_n(0)=0$ and  $g_n\in C^{2}(\RR^+_*)$, $n\in \NN$. There is also a sequence of 
$2\pi$-periodic $\rho$-trigonometrically convex positive functions $h_n\underset{n\to\infty}{\searrow} h$
(\cite[Proposition 1.4]{Kha91}, \cite[Theorem 51]{GM}) such that $h_n\in C^2(\RR^+_*)$, $n\in \NN$.
The limit of each decreasing sequence of positive subharmonic functions is a subharmonic positive function.
Therefore it suffices to prove the subharmonicity of the function \eqref{gh} on $D$ outward $\overline D(r_{\rho})$ for the case $h\in C^2(\RR)$ and  $g\in C^2(\RR^+_*)$.  The calculation of the Laplace operator of the  function \eqref{gh} in polar coordinates $(r,\theta)$ gives
\begin{multline}\label{Delta}
\Delta \left(g\Bigl(\frac{1-r}{r}\Bigr)h(\theta)\right)\overset{\eqref{gh}}{=}\Bigl(\frac{\partial}{\partial r^2} 
+\frac{1}{r}\frac{\partial}{\partial r}+\frac{1}{r^2}\frac{\partial}{\partial \theta}\Bigr)\bigl(g(1/r-1)h(\theta)\bigr) \\
=\Bigl(g''(1/r-1)\frac{1}{r^4}+\frac{1}{r^3}g'(1/r-1)\Bigr)h(\theta)+\frac{1}{r^2}g(1/r-1)h''(\theta).
\end{multline}
Each {\it convex  function\/ $g \colon \RR^+\to \RR^+$, $g\in C^2(\RR_*^+)$, with\/ $g(0)=0$ }
has the following properties:
\begin{equation}\label{gx}
g''\geq 0, \quad g'(x)\geq \frac{g(x)}{x} \quad\text{\it for all\/ $x\in \RR_*^+$, and $g\in C(\RR^+)$ is increasing.}
\end{equation}
It follows  from  \eqref{Delta}, \eqref{gx}, \eqref{C2tr} that, for $h\in \rtrc^+\cap \, C^2(\RR)$,
\begin{multline}\label{posgh}
\Delta \left(g\Bigl(\frac{1-r}{r}\Bigr)h(\theta)\right)\\
\overset{\eqref{gx}}{\geq} \Bigl(g''(1/r-1)\frac{1}{r^4}+\frac{1}{r^2(1-r)}g(1/r-1)\Bigr)h(\theta)
+\frac{1}{r^2}g(1/r-1)h''(\theta) \\
\overset{\eqref{gx}, \eqref{C2tr}}{\geq}
\frac{1}{r^2}\Bigl(\frac{1}{1-r}-\rho^2\Bigr)g(1/r-1) h(\theta)
\quad\text{for all $r\in \RR_*^{+}$, $\theta \in \RR$.}
\end{multline}
 If $r\geq r_{\rho}$, then the right-hand side of the inequalities 
\eqref{posgh}  is positive. Therefore the function \eqref{gh} is subharmonic on $\DD
\setminus \overline D(r_{\rho})$. Obviously, the function \eqref{gh} is positive, since the functions $h\in \rtrc^+$, 
$g\colon \RR^+\to \RR^+$ are positive, and 
 
\begin{equation}\label{g0}
 g (0) = 0 \quad \Longrightarrow \quad  \lim_{0<x\to 0}g(x)\overset{\eqref{gx}}{=}0 \quad \Longrightarrow \quad \lim_{1>r\to 1} g\Bigl(\frac{1-r}{r}\Bigr)h(\theta)\overset{\eqref{gh}}{=}0.
\end{equation}
Besides, in view of \eqref{gx}, we have 
\begin{equation*}
g\Bigl(\frac{1-r}{r}\Bigr)\max_{\theta}h(\theta)
 \overset{\eqref{rrho}}{\leq} 
g\Bigl(\frac{1-r_{\rho}}{r_{\rho}}\Bigr)\max_{\theta}h(\theta)
\overset{\eqref{hsbh}}{=} b_{\rho} \quad\text{for all $r\in (r_{\rho}, 1)$}.
\end{equation*}
So, by Definition \ref{def:tesf}, in view of \eqref{g0}, the function \eqref{gh} belongs to the class \eqref{hsbh}.  
\end{proof}

\section{Proofs of main results} 

\begin{proof}[of Main Theorem] Let $0\leq \rho \leq \sqrt{2}$. Then $r_{\rho}\overset{\eqref{rrho}}{=}1/2$.  
By Proposition \ref{trc-stf},  the function  \eqref{gh} belong to the class  $\sbh_0^+\bigl(\DD
\setminus \overline D(1/2);\leq 1\bigr)$, since $g(1)\leq 1$ and $\max_{\theta} h(\theta)\leq 1$ under the conditions [g]--[h] of Main Theorem. Hence, by Theorem B, there exists a constant $C=\const_{M,u}^+$ such that the inequality \eqref{mest} holds for any function $v$ of the form \eqref{gh}. Thus, we obtain 
\begin{multline}\label{uM+}
\int_{\DD\setminus \overline D(1/2)} g\Bigl(\frac{1-t}{t}\Bigr) h(\theta)\dd \mu_u(te^{i\theta})
\\
\overset{\eqref{mest}}{\leq} \int_{\DD\setminus \overline D(1/2)} g\Bigl(\frac{1-t}{t}\Bigr) h(\theta)\dd \mu_M(te^{i\theta})+C
\quad\text{\it for any\/ {\rm [g]}--{\rm [h]}}.
\end{multline}
\begin{lemma}[{\rm \cite{GS}, \cite{Kha91}--\cite{Kh12}}]\label{l1} Let $f$ be a continuous function 
on $(r,1)\subset  (0,1)$,  $\mu \in \Meas \DD$. Under the conditions  before \eqref{ntk} we have  the equality
\begin{equation}\label{uM+S}
\int_{\DD\setminus \overline D(r)} f(t) k(\theta)\dd \mu (te^{i\theta})\overset{\eqref{irR}}{=}
\int_{r}^1 f(t) \dd \mu^{\rad}(t;k).
\end{equation}
\end{lemma}
By Lemma \ref{l1}, we get from \eqref{uM+}  the conclusion  \eqref{uM} of Main Theorem for $\rho\leq \sqrt{2}$.

Consider now the case $\rho >\sqrt{2}$, i.\,e., $r_{\rho}\overset{\eqref{rrho}}{=}1-1/\rho^2>1/2$.  By Proposition \ref{trc-stf},  the function  \eqref{gh} belongs to the class  \eqref{hsbh}, and 
$$\sbh_0^+\bigl(\DD 
\setminus \overline D(r_{\rho});\leq b_{\rho}\bigr)\subset  \sbh_0^+\bigl(\DD
\setminus \overline D(r_{\rho});\leq 1\bigr)$$  
since $g(1)\leq 1$ and $\max_{\theta} h(\theta)\leq 1$ for {[g]}--{[h]}, and
\begin{equation*}
b_{\rho}\overset{ \eqref{hsbh}}{\leq} g\Bigl(\frac{1-(1-\rho^{-2})}{1-\rho^{-2}}\Bigr)\max_{\theta} h(\theta)\overset{\eqref{gx}}{\leq} g(1)\max_{\theta} h(\theta)\leq 1, 
\quad \text{when $\rho>\sqrt{2}$}.
\end{equation*}
Hence, by Theorem B, there is a constant $C'=\const_{S,M,u}^+=\const_{\rho,M,u}^+$ for $S:=\overline D(r_{\rho})$ such that the inequality \eqref{mest} holds for any function $v$ of the form \eqref{gh}. So, we get 
\begin{multline}\label{uM+r}
\int_{\DD\setminus \overline D(r_{\rho})} g\Bigl(\frac{1-t}{t}\Bigr) h(\theta)\dd \mu_u(te^{i\theta})
\\ \overset{\eqref{mest}}{\leq} 
\int_{\DD\setminus \overline D(r_{\rho})} g\Bigl(\frac{1-t}{t}\Bigr) h(\theta)\dd \mu_M(te^{i\theta})+C'
\quad\text{\it for any\/ {\rm [g]}--{\rm [h]}}.
\end{multline}
It is easy to see that there are constants $C'':=\const_{\rho, u}^+$, $C''':=\const_{\rho, M}^+$  such that
\begin{subequations}
\begin{align*}
\int_{D(r_{\rho})\setminus \overline D(1/2)} 
g\Bigl(\frac{1-t}{t}\Bigr) h(\theta)\dd \mu_u(te^{i\theta})\leq \mu_u\bigl(D(r_{\rho})\setminus \overline D(1/2)\bigr)&\leq C'', \\
\Bigl|\int_{D(r_{\rho})\setminus \overline D(1/2)} 
g\Bigl(\frac{1-t}{t}\Bigr) h(\theta)\dd \mu_M(te^{i\theta})\Bigr|\leq |\mu_M|\bigl(D(r_{\rho})\setminus \overline D(1/2)\bigr)&\leq C''',
\end{align*}
\end{subequations}
Hence, in view of \eqref{uM+r}, we obtain \eqref{uM+} with $C:=C'+C''+C'''= \const_{\rho,M,u}$
for any\/ {\rm [g]}--{\rm [h]}. 
By Lemma 1, we again obtain from \eqref{uM+}  the conclusion  \eqref{uM} of Main Theorem already for the case $\rho>\sqrt{2}$.

Let\/ ${\sf Z}$ be a sequence   from\/ \eqref{dfZ} with $\mathcal O:=\DD$, and let 
$f\in \Hol_*(\DD)$ be a function  that vanishes on the sequence $\sf Z\subset \Zero_f$ 
and satisfies the inequality $u:=\log |f|\leq M$.\/ By  the conclusion  \eqref{uM} of Main Theorem, there exists a constant $C:=\const_{\rho,M,f}^+$ such that we have  \eqref{uM} for $u:=\log |f|$. 
Here the choice of the function  $f$ is predetermined solely by the sequence $\sf Z$ and function $M$.\/  So, $C=\const_{\rho,M,{\sf Z}}^+$, and, in view of  the (in)equalities
\begin{multline*}
\sum_{1/2< r_k<1} g\Bigl(\frac{1-r_k}{r_k}\Bigr) h(\theta_k)
\overset{\eqref{Ztf}}{=}\int_{1/2}^{1} g\Bigl(\frac{1-t}{t}\Bigr) \dd n_{\sf Z}^{\rad} (t;h)
\\
\leq \int_{1/2}^{1} g\Bigl(\frac{1-t}{t}\Bigr) \dd n_{\Zero_f}^{\rad} (t;h)
\overset{\eqref{muf}}{=} \int_{1/2}^{1} g\Bigl(\frac{1-t}{t}\Bigr) \dd \mu_{u}^{\rad} (t;h)
\quad \text{for\/ $u:=\log|f|$}, 
\end{multline*}
the inequality \eqref{uMf} follows from \eqref{uM}. 
\end{proof}

\begin{proof}[of Uniqueness Theorem] Without loss of generality, we can assume that $h\not\equiv  0$, i.\,e., $h_0:=\max_{\theta}h(\theta)>0$,  and  $g(1)>0$. 
If $f\in \Hol_*(\DD)$, i.\,e., $f\neq 0$, $f({\sf Z})=0$, and $|f|\leq \exp M$ on $\DD$, then, by 
Main Theorem, we have 
\begin{multline*}\label{uM2}
\frac{1}{g(1)h_0}\sum_{1/2< r_k<1} g(1-r_k) h(\theta_k) 
\leq 
\sum_{1/2< r_k<1} \frac{1}{g(1)}g\Bigl(\frac{1-r_k}{r_k}\Bigr) \frac{1}{h_0}h(\theta_k) 
\\
\overset{\eqref{uMf}}{\leq} 
\int_{1/2}^{1} \frac{1}{g(1)}g\Bigl(\frac{1-t}{t}\Bigr) \dd \mu^{\rad}_M(t;h/h_0)
\\
\leq \frac{1}{g(1)h_0}\int_{1/2}^{1} g\bigl(2(1-t)\bigr) \dd \mu^{\rad}_M(t;h)\overset{\eqref{cuM}}{<}+\infty.
\end{multline*}
So, if $f\neq 0$, then the latter contradicts the condition \eqref{cuZ}.
\end{proof}

\begin{acknowledgements}
The authors thank the organizers of International Conferences 
``Complex Analysis and Related Topics 2018'' (April 23--27, 2018, Euler International Mathematical Institute, St. Petersburg, Russia) and ``XXVII St.Petersburg Summer Meeting in Mathematical Analysis''
(August 6--11, 2018, St. Petersburg, Russia) for the invitation and for the opportunity to report the results related to the content of this article.
\end{acknowledgements}



\end{document}